\documentclass[10pt,a4paper]{amsart}

\usepackage{amsmath}
\usepackage{amssymb}
\usepackage{amsthm}
\usepackage{cleveref}
\usepackage[shortlabels]{enumitem}
\usepackage{mathtools}

\newtheorem{corollary}{Corollary}[section]
\newtheorem{definition}[corollary]{Definition}
\newtheorem{lemma}[corollary]{Lemma}

\newtheorem{theorem}[corollary]{Theorem}

\DeclareMathOperator*{\cl}{cl}

\DeclareMathOperator*{\lin}{lin}
\DeclareMathOperator*{\Rea}{Re}

\def\p{\partial}
\def\C{{\mathbb C}}
\def\K{{\mathbb K}}
\def\N{{\mathbb N}}

\def\R{{\mathbb R}}
\def\LL{{\mathcal L}}

\begin{document}
	
	\title[Separability and Submetrizability]
	{Separability and Submetrizability in Locally Convex Spaces}
	
	\author[Thomas Ruf]{Thomas Ruf}
	\address{%
		Institut für Geometrie\\
		TU Dresden\\
		01069 Dresden, Germany}
	\email{thomas.ruf@tu-dresden.de, thomas.ruf.math@web.de}
	
	
	\begin{abstract}
		We introduce the property of countable separation for a locally convex Hausdorff space $X$ and relate it to the existence of a metrizable coarser topology. Building on this, we demonstrate how the separability of $X$ is equivalent to the existence of a locally convex topology on the dual $X'$ that is metrizable and coarser than the weak topology $\sigma(X', X)$. This result generalizes known conditions for separability and provides a precise duality between separability and metrizability. We also show how to derive new and known conditions for the separability of $X$ from this characterization.
	\end{abstract}
	
	\maketitle
	
	\section{Introduction}
	A topological space is said to be separable if it contains a countable dense subset. This property enables the entire space to be approximated by a relatively simple set, making separability a highly desirable property. It is therefore natural to ask whether a given space is separable.
	
	This note provides a new equivalent description of separability for locally convex Hausdorff vector spaces: such a space $X$ is separable if and only if there is a locally convex topology on its continuous dual $X'$ that is metrizable and coarser than the weak topology $\sigma(X', X)$. We refer to this condition by saying that $\sigma(X', X)$ is locally convexly submetrizable. Since $X'$ is also a locally convex Hausdorff space, the dual statement holds symmetrically: $\sigma(X, X')$ is locally convexly submetrizable if and only if $X'$ is separable.
	
	Importantly, the particular choice of the locally convex topology on $X'$ is immaterial for separability as long as the topology is compatible with the duality of $(X, X')$ in the sense of the Mackey-Arens theorem. However, this indifference fails with respect to the topology under which $X$ itself is locally convexly submetrizable. For example, any Banach space whose dual is not weak* separable is norm-metrizable without being submetrizable in its weak topology. Nevertheless, we proceed to show that this discrepancy disappears when $X$ is separable. Accordingly, a separable locally convex Hausdorff space $X$ is locally convexly submetrizable if and only if its dual $X'$ is separable.
	
	Our result has some partial predecessors. For example, \cite[Ex. 8.201, p. 271]{NB} suggests that there is a 'certain duality' between separability and metrizability for locally convex spaces. The reader is then instructed to verify several lemmas that can be easily obtained as corollaries of our main result. Moreover, the easy implication that separability of $\sigma(X', X)$ implies submetrizability of $X$ is stated in \cite[Prop. 2.5.18]{PB}. In its present precision and completeness, however, our main result seems to be new.
	
	All vector spaces in this paper are over the scalar field $\K \in \left\{ \C, \R \right\}$.
	
	\section{Results} \label{sec:res}
	In this section, we give an equivalent characterization of when a locally convex Hausdorff vector space $X$ is separable in terms of a countable separation property and submetrizability of the dual space in its weak topology. Moreover, we prove that the weak topology in this characterization may be replaced by any other locally convex topology on $X'$ that is compatible with the duality of $(X, X')$ if $X'$ is separable. For a set $A \subset X$, we denote by $A^\circ \subset X'$ its polar set, i.e.,
	\begin{equation}
		A^\circ = \left\{ x' \in X' \colon \Rea \langle x', x \rangle \le 1 \quad \forall x \in A \right\}.
	\end{equation}
	
	\begin{definition} \label{def:cntbl sep}
		Let $X$ be a locally convex Hausdorff vector space. We call $X$ \emph{countably separated} if there is a sequence of (absolutely convex, closed) null neighborhoods $U_n$ with $\bigcap_{n \ge 1} \left\{ x \in X \colon S(U_n^\circ, x) = 0 \right\} = \left\{ 0 \right\}$. Here, $U_n^\circ$ is the polar set of $U_n$ and $S(U_n^\circ, x) = \sup_{x' \in U_n^\circ} \Rea \langle x', x \rangle$.
	\end{definition}
	
	By the bipolar theorem \cite[Satz VIII.3.9]{We}, the bipolar set $U_n^{\circ \circ}$ agrees with the absolutely convex, closed envelope of $U_n$ and $U_n^\circ = U_n^{\circ \circ \circ}$ so that it is immaterial in \Cref{def:cntbl sep} whether $U_n$ is absolutely convex or closed. Geometrically, it is interesting to note that if $A \subset X$, then the set
	$$
	\left\{ x \in X \colon S(A^\circ, x) = 0 \right\}
	$$
	is a closed, linear subspace since $s(A, \cdot)$ is absolutely homogeneous and subadditive. It is the closed line space of $A^{\circ \circ}$, i.e., the closure of the maximal linear subspace contained in $A^{\circ \circ}$.
	
	\begin{theorem} \label{thm:subm}
		A locally convex Hausdorff vector space $X$ is countably separated if and only if there is a coarser, metrizable topology $\rho$ on $X$. In this case, it is possible to choose $\rho$ locally convex.
	\end{theorem}
	
	\begin{proof}
		$\implies$: let $(X, \tau)$ be countably separated by a sequence of absolutely convex, closed null neighborhoods $U_n$ with
		$$
		\bigcap_{n \ge 1} \left\{ x \in X \colon S(U_n^\circ, x) = 0 \right\} = \left\{ 0 \right\}.
		$$
		Let $\rho$ be the locally convex topology whose null neighborhood base is given by the countable collection $\left\{ U_n \right\}$. This is equivalent to defining the topology by requiring the $\rho$-continuity of the seminorms $p_n(x) = \inf\left\{ \lambda > 0 \colon \lambda x \in U_n \right\}$. A metric inducing $\rho$ is given by
		$$
		d(x_1, x_2) = \sum_{n \ge 1} 2^{-n} \frac{p_n(x_1 - x_2)}{1 + p_n(x_1 - x_2)},
		$$
		which is $\tau$-continuous, i.e., $\rho$ is coarser than $\tau$.
		\\
		$\impliedby$: let $\rho$ be a metrizable topology with $\rho \subset \tau$. By metrizability, there is a sequence $U_n$ forming a null $\rho$-neighborhood base. Since every $U_n$ is also a $\tau$-neighborhood and $\tau$ is locally convex, there exist absolutely convex, closed null neighborhoods $V_n \subset U_n$ by \cite[Satz VIII.1.5]{We}. Invoking the bipolar theorem \cite[Satz VIII.3.9]{We} and using that $U_n$ is a null neighborhood base of a Hausdorff topology, we conclude
		$$
		\left\{ 0 \right\} \subset \bigcap_{n \ge 1} \left\{ x \in X \colon S(V_n^\circ, x) = 0 \right\} \subset \bigcap_{n \ge 1} V_n^{\circ \circ} = \bigcap_{n \ge 1} V_n = \left\{ 0 \right\}
		$$
		so that $(X, \tau)$ is countably separated.
	\end{proof}
	
	To prepare the important results \Cref{thm:sep} and \Cref{thm:main}, we need the following lemmas, the first of which is well-known. Its proof is omitted.
	
	\begin{lemma} \label{lem:hull dens}
		Let $X$ be a topological vector space over $\K$, $\K$ being $\C$ or $\R$. Then $X$ is separable if and only if there exists a sequence $\left\{ x_1, \ldots, x_n, \ldots \right\}$ in $X$ such that the linear subspace generated by $\left\{ x_1, \ldots, x_n, \ldots \right\}$ is dense.
	\end{lemma}
	
	\begin{lemma} \label{lem:pol intersec}
		Let $X$ be a locally convex vector space and $C_i \subset X$ be sets satisfying $\bigcap_i C_i^{\circ \circ} = \left( \bigcap_i C_i \right)^{\circ \circ}$. This assumption is satisfied, e.g., if each $C_i$ is a closed convex set containing the origin. Then the polar set of $C = \bigcap_i C_i$ is given by
		\begin{equation} \label{eq:pol conc}
			C^\circ = \overline{\mathrm{co} } \bigcup_i C_i^\circ.
		\end{equation}
		Here, $\overline{\mathrm{co} }$ denotes the closed convex set generated by another set.
	\end{lemma}
	
	\begin{proof}
		By invoking \cite[Lem. VIII.3.8(a) and (e)]{We} and our assumption, we compute
		\begin{equation} \label{eq:prepol id}
			\left( \overline{\mathrm{co} } \bigcup_i C_i^\circ \right)^\circ
			= \left( \bigcup_i C_i^\circ \right)^\circ
			= \bigcap_i C_i^{\circ \circ}
			= \left( \bigcap_i C_i \right)^{\circ \circ}
			= C^{\circ \circ}.
		\end{equation}
		By taking the polar of identity \eqref{eq:prepol id} and using the bipolar theorem \cite[Satz VIII.3.9]{We} on both sides of the resulting identity, we conclude \eqref{eq:pol conc}.
	\end{proof}
	
	\begin{theorem} \label{thm:sep}
		Let $X$ be a locally convex Hausdorff vector space. The following are equivalent:
		\begin{enumerate}[(i)]
			
			\item $X$ is separable; \label{it:en sep1}
			
			\item $X'$ is countably $\sigma(X', X)$-separated; \label{it:en sep2}
			
			\item There is a metrizable topology on $X'$ that is coarser than $\sigma(X', X)$. \label{it:en sep3}
			
			\item There is a locally convex, metrizable topology on $X'$ that is coarser than $\sigma(X', X)$. \label{it:en sep4}
			
		\end{enumerate}
	\end{theorem}
	
	\begin{proof}
		\Cref{thm:subm} implies the equivalences $\ref{it:en sep2} \iff \ref{it:en sep3} \iff \ref{it:en sep4}$. $\ref{it:en sep1} \implies \ref{it:en sep2}$: Let $\left\{ x_1, x_2, \ldots \right\} \subset X$ be a dense countable subset. The bipolar theorem \cite[Satz VIII.3.9]{We} implies $\left\{ x_1, x_2, \ldots \right\}^\circ = \left\{ 0 \right\}$. Thus, defining a sequence of absolutely convex, closed null neighborhoods in $\sigma(X', X)$ by setting
		$$
		U_n = \left\{ x' \in X' \colon \left| \langle x', x_n \rangle \right| \le 1 \right\},
		$$
		the bipolar theorem gives
		$$
		\left\{ 0 \right\} \subset \bigcap_{n \ge 1} \left\{ x \in X \colon S(U_n^\circ, x) = 0 \right\} \subset \bigcap_{n \ge 1} U_n^{\circ \circ} = \bigcap_{n \ge 1} U_n = \left\{ 0 \right\}.
		$$
		\\
		$\ref{it:en sep1} \impliedby \ref{it:en sep2}$: Let $U_n$ be null neighborhoods in $\sigma(X', X)$ with
		$$
		\bigcap_{n \ge 1} \left\{ x \in X \colon S(U_n^\circ, x) = 0 \right\} = \left\{ 0 \right\}.
		$$
		By definition of the weak topology $\sigma(X', X)$, a neighborhood base is given by (absolutely convex and closed) sets having the form
		$$
		V = \bigcap_{i = 1}^k \left\{ x' \in X' \colon \left| \langle x', x_{i, V} \rangle \right| \le 1 \right\}, \quad k \in \N.
		$$
		Since a set of the form $\left\{ x' \in X' \colon \left| \langle x', x_{i, V} \rangle \right| \le 1 \right\}$ has as its polar the absolutely convex set $\mathrm{aco} \left\{ x_{i, V} \right\}$ generated by the vector $x_{i, V}$, an invocation of \Cref{lem:pol intersec} shows that
		\begin{equation} \label{eq:V polar}
			V^\circ = \mathrm{aco} \left\{ x_{1, V}, x_{2, V}, \ldots, x_{k, V} \right\}.
		\end{equation}
		We pick such a set $V_n \subset U_n$ for every $n \in \N$. Then
		$$
		\left\{ 0 \right\} \subset \bigcap_{n \ge 1} \left\{ x' \in X' \colon S(V_n^\circ, x') = 0 \right\} \subset \bigcap_{n \ge 1} \left\{ x' \in X' \colon S(U_n^\circ, x') = 0 \right\} = \left\{ 0 \right\}.
		$$		
		Consequently, if $\Rea \langle x', x_{i, V_n} \rangle = 0$ for all $1 \le i \le k(n)$ and $n \in \N$, which is equivalent to
		$$
		x' \in \bigcap_{n \ge 1} \left\{ x' \in X' \colon S(V_n^\circ, x') = 0 \right\}
		$$
		by \eqref{eq:V polar}, then	$x' = 0$, so that the countable collection $\left\{ x_{i, V_n} \right\}$ has a trivial polar set, i.e., its absolutely convex hull is dense in $X$ by the bipolar theorem. We conclude the separability of $X$ by \Cref{lem:hull dens}.
	\end{proof}
	
	\begin{corollary} \label{cor:sep}
		Let $X$ be a locally convex Hausdorff vector space. The following are equivalent:
		\begin{enumerate}[(i)]
			
			\item $X'$ is (weakly) separable; \label{it:en corsep1}
			
			\item $X$ is countably $\sigma(X, X')$-separated; \label{it:en corsep2}
			
			\item There is a metrizable topology on $X$ that is coarser than $\sigma(X, X')$. \label{it:en corsep3}
			
			\item There is a locally convex, metrizable topology on $X$ that is coarser than $\sigma(X, X')$. \label{it:en corsep4}
			
		\end{enumerate}
	\end{corollary}
	
	\begin{proof}
		If $X$ is a locally convex Hausdorff space, its dual $X'$ shares this property so that the theorem still applies if the role of $X$ in \Cref{thm:sep} is taken by $\left( X', \tau \right)$, where $\tau$ is the weak topology $\sigma(X', X)$ or any other locally convex topology that is compatible with the same duality. Note in this regard that $\left( X'_\tau \right)' = X$.
	\end{proof}
	
	In contrast to separability, if a space is countably separated, then it need not be so in every other locally convex topology that is compatible with the given duality. For example, let $X$ be a non-separable, reflexive Banach space. Then $X$ is countably separated in norm topology, but $X'$ is not separable. Thus, by \Cref{thm:sep}, the weak topology on $X$ cannot be countably separated despite being compatible with the duality provided by the norm topology. However, the dependence on the particular topology for countable separation may be avoided if the locally convex space $X$ is separable.
	
	\begin{theorem}	\label{thm:main}
		Let $X$ be a separable locally convex Hausdorff vector space. If $X$ is countably separated, then $X'$ is (weakly) separable.
	\end{theorem}
	
	\begin{proof}
		Since a convex set is closed iff it is closed in the Mackey topology $\mu(X, X')$ by \cite[Satz VIII.3.5]{We} and the Mackey-Arens theorem, we may even assume that $X$ is separable in the Mackey topology $\mu(X, X')$. Suppose $U$ is an absolutely convex, closed, $\mu(X, X')$-neighborhood. The Banach-Alaoglu theorem \cite[Thm. VIII.3.11]{We} implies that the polar set
		$$
		U^\circ = \left\{ x' \in X' \colon \Rea \langle x', x \rangle \le 1 \quad \forall x \in U \right\}
		$$
		is compact in the weak topology $\sigma(X', X)$. Therefore, the support function $S(U^\circ, x) \coloneqq \max_{x' \in U^\circ} \Rea \langle x', x \rangle$ for $x \in X$, is $\mu(X, X')$-continuous by definition of $\mu(X, X')$ as the topology generated by the support functions of absolutely convex, weakly compact sets. The set $U$ being absolutely convex and closed, the bipolar theorem \cite[Satz VIII.3.9]{We} shows that $U = U^{\circ \circ}$, which is the same as
		$$
		U = \left\{ x \in X \colon S(U^\circ, x) \le 1 \right\}.
		$$
		In particular, we see that
		\begin{equation} \label{eq:bdry desc}
			\p U = \left\{ x \in X \colon S(U^\circ, x) = 1 \right\}.
		\end{equation}
		Remember that the intersection of a dense set with an open set in a topological space is dense in the open set. In particular, if $X$ is separable, then $X \setminus \left\{ x \in X \colon S(U^\circ, x) = 0 \right\}$ is so, too. Moreover, taking continuous images preserves denseness. Thus, by considering the $\mu(X, X')$-continuous mapping
		$$
		P_U \colon X \setminus \left\{ S(U^\circ, \cdot) = 0 \right\} \to \p U \colon x \mapsto \frac{x}{S(U^\circ, x) },
		$$
		which projects its domain onto $\p U$, we see that $\p U$ is separable:
		\begin{equation} \label{eq:den seq}
			\text{There is a sequence } \left\{ x_1, x_2, \ldots \right\} \text{ that is dense in } \p U.
		\end{equation}
		Since $X$ is assumed to be countably separated, there exists a sequence of absolutely convex closed neighborhoods of the origin $U_n \subset X$ satisfying
		$$
		\bigcap_{n \ge 1} \left\{ x \in X \colon S(U_n^\circ, x) = 0 \right\} = \left\{ 0 \right\}.
		$$
		By the Mackey-Arens theorem, $\mu(X, X')$ is the finest locally convex topology that is compatible with the given duality. Therefore, every neighborhood $U_n$ is a $\mu(X, X')$-neighborhood. Hence, we can apply our consideration \eqref{eq:den seq} to every neighborhood $U_n$ to obtain a sequence $\left\{ x_{1, n}, x_{2, n}, \ldots \right\}$ that is dense in the boundary $\p U_n$. By definition of the polar set, there exist $x'_{m, n} \in U_n^\circ$ with
		$$
		\Rea \langle x'_{m, n}, x_{m, n} \rangle > \frac{1}{2} \quad \forall (n, m) \in \N \times \N.
		$$
		Set
		$$
		\LL = \lin \left\{ x'_{1, 1}, x'_{1, 2}, x'_{2, 2}, x'_{1, 3}, x'_{2, 3}, x'_{3, 3}, \ldots \right\}.
		$$
		To finish the proof, it suffices to show that $\LL$ is dense in $X'$ by \Cref{lem:hull dens}. Suppose that $x \in X$ with $\left. x \right|_\LL = 0$. If $x \ne 0$, then countable separation implies that there exists $U_n$ with $S(U_n^\circ, x) > 0$ and hence we may replace $x$ with its projection $P_{U_n} x$ onto $\p U_n$. By denseness, there is a subnet of the sequence $\left\{ x_{1, n}, x_{2, n}, \ldots \right\}$ converging to $x$. In particular, the sequence of real numbers
		$$
		m \mapsto S \left( U_n^\circ, x_{m, n} - x \right)
		$$
		has zero as an accumulation point by the continuity of $S(U_n^\circ, \cdot)$ in $\mu(X, X')$. Therefore, remembering \eqref{eq:bdry desc}, we find $x_{m_0, n}$ with $S(U_n^\circ, x_{m_0, n} - x ) \le \tfrac{1}{4}$. Consequently, we arrive at the contradiction
		$$
		\frac{1}{2} 
		< \Rea \langle x'_{m, n}, x_{m, n} \rangle
		= \Rea \langle x'_{m, n}, x_{m, n} - x \rangle
		\le S \left( U_n^\circ, x_{m_0, n} - x \right)
		\le \tfrac{1}{4}.
		$$
		Hence, we have to drop our contradictory assumption, implying $x = 0$. In conclusion, the bipolar theorem implies the weak denseness $\cl \LL = X'$.
	\end{proof}
	
	Now, we have the following corollaries. \Cref{cor:main} is new. \Cref{cor:sep sub} generalizes a corresponding result where metrizability instead of submetrizability is required, whereas \Cref{cor:sep ball} and \Cref{cor:sep dual} are known in the precise form stated here.
	
	\begin{corollary} \label{cor:main}
		Let $X$ be a separable locally convex Hausdorff vector space and $S \subset X'$ be a linear separating subspace. Suppose that $X$ is countably separated by a sequence of absolutely convex $\mu(X, S)$-closed null neighborhoods. Then $S$ is separable and (weakly) dense in $X'$. In particular, $X'$ itself is separable.
	\end{corollary}
	
	\begin{proof}
		We apply \Cref{thm:main} to the dual pair $(X, S)$ to find that $S$ is $\sigma(S, X)$-separable. The claim follows upon observing that, by the bipolar theorem \cite[Satz VIII.3.9]{We}, $S$ is (weakly) dense in $X'$ since it separates points.
	\end{proof}
	
	\begin{corollary} \label{cor:sep sub}
		Let the locally convex Hausdorff vector space $X$ be separable and submetrizable. Then $X'$ is separable.
	\end{corollary}
	
	\begin{proof}
		Combine \Cref{thm:subm} with \Cref{thm:main}.
	\end{proof}
	
	\begin{corollary} \label{cor:sep ball}
		A normed space $X$ is separable if and only if the closed unit ball of its dual $B_{X'}$ is weak* metrizable.
	\end{corollary}
	
	\begin{proof}
		$\implies$: by \Cref{thm:sep}, a normed space $X$ is separable if and only if the weak* topology $\sigma(X', X)$ is submetrizable. Let $\rho$ be the metrizable topology that is coarser than $\sigma(X', X)$. The identical embedding $\left( B_{X'}, \sigma(X', X) \right) \to \left( B_{X'}, \rho \right)$ has a continuous inverse since it maps continuously from a compact space into a Hausdorff space. Consequently, the weak* topology agrees with $\rho$ on $B_{X'}$ and hence is metrizable.
		\\
		$\impliedby$: if $B_{X'}$ is weak* metrizable, then there is a countable neighborhood base of zero in $B_{X'}$. Every base element is given by the intersection of some neighborhood in $\sigma(X', X)$ with $B_{X'}$. By possibly decreasing each base element, we may assume that the neighborhoods in $\sigma(X', X)$ are absolutely convex and closed. Since their (absolutely convex) intersection does not meet $B_{X'}$ except at the origin, these neighborhoods separate $\sigma(X', X)$ countably. We conclude that $X$ is separable by \Cref{thm:sep}.
	\end{proof}
	
	\begin{corollary} \label{cor:sep dual}
		A normed space $Y$ is separable if the dual $Y'$ is norm separable.
	\end{corollary}
	
	\begin{proof}
		In \Cref{cor:main}, consider $X = Y'$ and $S = Y \subset Y''$. Then $X$ is countably separated by its $\sigma(X, S)$-closed unit ball so that $S$ is weakly dense in $X'$. Since the trace of the weak* topology $\sigma(Y'', Y')$ agrees with the weak topology $\sigma(Y, Y')$ on $Y$, we conclude that $Y$ is weakly separable. Therefore, it is norm separable by \cite[Satz VIII.3.5]{We}.
	\end{proof}

	This version of the article has been accepted for publication, after peer review but is	not the Version of Record and does not reflect post-acceptance improvements, or any corrections. The Version of Record is available online at: https://doi.org/10.1007/s43034-026-00524-x

\end{document}